\documentclass[a4paper]{amsproc}
\usepackage{amssymb,faktor}
\usepackage{mathtools}
\usepackage{tensor}
\usepackage[hyphens]{url} 
\usepackage[margin=0.9in]{geometry}
\usepackage[hyperfootnotes=false,colorlinks=true,allcolors=blue]{hyperref}

\theoremstyle{plain}
\newtheorem{thm}{Theorem}[section]

\newtheorem{cor}{Corollary}[section]
\theoremstyle{definition}

\theoremstyle{remark}

\newtheorem*{rem*}{Remark}
\newtheorem*{ack*}{Acknowledgements}
\numberwithin{equation}{section}

\renewcommand{\leq}{\leqslant}
\renewcommand{\geq}{\geqslant}
\numberwithin{equation}{section}
\newtheorem{theorem}{Theorem}[section]

\newtheorem{lemma}[theorem]{Lemma}

\newtheorem*{lemma*}{Lemma}
\newtheorem{claim}{Claim}
\newtheorem*{claim*}{Claim}
\newcommand{\thistheoremname}{}
\newtheorem{genericthm}[thm]{\thistheoremname}
\makeatletter
\def\blfootnote{\xdef\@thefnmark{}\@footnotetext}
\makeatother

\newcommand\numberthis{\addtocounter{equation}{1}\tag{\theequation}}

\newcommand{\Z}{\ensuremath{\mathbb{Z}}}

\newcommand{\GL}{\ensuremath{\text{GL}}}
\newcommand{\GS}{\ensuremath{\text{GSpin}}}
\newcommand{\SO}{\ensuremath{\text{SO}}}
\newcommand{\SL}{\ensuremath{\text{SL}}}
\newcommand{\ind}{\text{Ind}}
\newcommand{\inv}{\ensuremath{\prescript{\iota}{}}}

\title[Stability of local gamma factors for GSpin groups]{Stability of local gamma factors arising from the doubling method for general spin groups}

\begin{document}

	
	\author{Siddhesh Wagh}
	\begin{abstract}
		In this work we prove that the local $\gamma$-factor arising from the doubling integrals for split general spin groups is stable. This deep property of the $\gamma$-factor constitutes an important ingredient in the application of the (generalized) doubling method to the construction of a global functorial lift. We obtain our result by adapting the arguments of Rallis and Soudry who proved the stability property for symplectic and orthogonal groups.
	\end{abstract}
	
	\maketitle
	
	\section{Introduction}
	
	Let $F$ be a local non-archimedean field of characteristic $0$, $\psi$ be a nontrivial additive character of $F$, and
	$G$ be the split general spin group of even or odd rank, defined over $F$. 
	Consider an irreducible representation $\pi$ of $G$, and a quasi-character $\tau$ of $F^\times$.
	The doubling method for $\pi\times\tau$ was developed in \cite{CFGK} following the classical construction of Piatetski-Shapiro and Rallis \cite{PSR}, with extension to the general spin group in \cite{CFK, GK}. In this work we obtain the stability result of the local $\gamma$-factor $\gamma(s,\pi\times\tau,\psi)$ defined in \cite{CFK, GK}.
	\blfootnote{The author was supported by the ISRAEL SCIENCE FOUNDATION grant No. 421/17.}
	\begin{thm} \label{thm:1.1}
		The local gamma factor $\gamma(s,\pi\times\tau,\psi)$ is stable, for $\tau$ sufficiently ramified. In detail, given two representations $\pi_1$ and $\pi_2$ that agree on the connected component of the center, there exists a positive integer $N$ such that 
		\begin{equation*}
			\gamma(s,\pi_1\times\tau,\psi) = \gamma(s,\pi_2\times\tau,\psi)
		\end{equation*}
		for all quasi-characters $\tau$ whose conductor has an exponent larger than $N$. 
	\end{thm}
	
	The stability results for generic representations of split classical groups are known due to the works of Cogdell, Piatetski-Shapiro \cite{CPS} for Rankin-Selberg method, and Cogdell \textit{et al} \cite{CPSS} using Landlands-Shahidi method for local coefficients. For non-generic representations the results are due to Rallis and Soudry \cite{RS}. Similar results for the symplectic and unitary groups over quadratic extensions were proved by Zhang in \cite{ZH}.  
	
	The actual construction of \cite{CFGK,CFK,GK}, called the generalized doubling method, included irreducible generic representations $\tau$ of $\mathrm{GL}_k$ for all $k\geq1$. The main motivation was to obtain a new proof of global functoriality (to the appropriate $\mathrm{GL_N}$), which is independent of the trace formula and its prerequisites, and relies solely on an integral representation and Converse Theorem. Using the multiplicativity property of the $\gamma$-factor with respect to $\tau$ (proved in \cite[Theorem~27]{CFK}), the stability result for $\gamma(s,\pi\times \tau,\psi)$ was reduced to the case proved here, i.e. to $k=1$ (see \cite[Lemma~51]{CFK}).
	
	To obtain our proof we follow the arguments of Rallis and Soudry \cite{RS}, who proved the stability of the $\gamma$-factor for doubling integrals in the context of orthogonal and symplectic groups. The difficulties here are technical in nature and related to the complications one encounters when dealing with general spin groups. 
	
	We also mention that the definitive treatment of the local doubling method for the case $k=1$ is due to Lapid and Rallis \cite{LR}, with several additional groups added in \cite{GAN,YAM} (at the time, general spin groups were not included).
	
	\begin{ack*}
		I would like to thank Eyal Kaplan for introducing me to this problem and for the numerous valuable, encouraging and inspiring discussions. 
	\end{ack*}

	\section{Basic notation}
	\subsection{Special orthogonal groups}
	Let $F$ be a local non-archimedean field of characteristic 0. We denote by $\mathcal{O}_F$ its ring of integers and let $\mathcal{P}_F$ be the prime ideal of $\mathcal{O}_F$. Assuming the cardinality of $\mathcal{O}_F / \mathcal{P}_F$ is $q$, we denote by $| \cdot |$ the absolute value on $F$, with $|\varpi| = q^{-1}$ for a generator $\varpi$ of $\mathcal{P}_F$.
	
	Let $V$ be a $c$-dimensional vector space over $F$ with a symmetric non-degenerate bilinear form $b$ and an orthogonal basis $B = (v_1,\ldots,v_c)$ such that $b(v_i,v_j)=1$ if $i+j = c+1$ and $0$ otherwise. Let $G_0$ be the group of symmetries of $(V,b)$. Then $G_0$, written as a matrix group with the basis $B$, is isomorphic to the group 
	\begin{equation*}
		O_m(F) = \{g \in \GL_c(F)|\prescript{\mathrm t}{}g J g = J \}
	\end{equation*}
	where $\prescript{\mathrm t}{}g$ is the transpose of $g$ and 
	$$ J = \begin{bmatrix}
		& & 1\\
		& . \cdot^. &  \\
		1 & &
	\end{bmatrix}.$$
	
	We denote by $G_0^0$ the connected component of the identity which is isomorphic to 
	\begin{equation*}
		\SO_c(F) = \{g \in O_c(F)| \det(g) = 1\}
	\end{equation*}
	and the Lie algebra of $G_0$ is denoted by 
	\begin{equation}
		\mathfrak{g} =  \mathfrak{so}_c(F) = \{x \in \text{Mat}_c(F)| \prescript{\mathrm t}{} x J + Jx = 0\}. \label{2.3}
	\end{equation}
	We denote the Borel subgroup of upper triangular elements of $\SO_c$ by $B_{\SO_c}$ and the torus of the diagonal elements by $T_{\SO_c}$. 
	
	\subsection{Split general spin groups}
	We shall first define the group Spin$_c$ which is the algebraic double cover of SO$_c$ following the description from \cite{CFK}. We denote by $B_{\text{Spin}_c}$ the Borel subgroup of Spin$_c$ which is the preimage of $B_{\text{SO}_c}$. Let $\epsilon_i$ denote the pullback of the $i$-th coordinate function from $T_{\text{SO}_c}$ to $T_{\text{Spin}_c}$, the torus of Spin$_c$. Define $\epsilon_j^\vee $ such that $\langle\epsilon_i,\epsilon_j^\vee \rangle = \delta_{i,j}$, where $\langle, \rangle$ is the standard pairing. The set of simple roots of Spin$_c$ is $\Delta_c = \{\alpha_0,\ldots,\alpha_{n-1}\}$ with $n = \lfloor c/2\rfloor$. Here $\alpha_i = \epsilon_i - \epsilon_{i+1}$ for $0 \leq i < n-1$, $\alpha_{n-1} = \epsilon_{n-1} +\epsilon_n$ if $c$ is even and $\alpha_{n-1} = \epsilon_{n-1}$ otherwise. 
	
	We will identify the split general spin group $G = \text{GSpin}_c$ with the Levi subgroup of Spin$_{c+2}$ obtained by removing $\alpha_0$ from $\Delta_{c+2}$. This fixes the Borel subgroup $B_G$ of GSpin$_c$. Note that Spin$_c$ is the derived group of GSpin$_c$, GSpin$_0$ = GSpin$_1$ = GL$_1$ and GSpin$_2$ = GL$_1$ $\times$ GL$_1$. Define a ``canonical'' character $\Upsilon$ of GSpin$_c$ as the lift of $-\epsilon_0$. Let
	\begin{equation*}
		C^0_G = \{\mathfrak{r}^\vee_c(t): t \in F^\times\}, \quad \mathfrak{r}_c = 
		\begin{cases}
			\alpha_0 \quad \hfill c=0, \\
			\alpha_0 + \alpha_1 \quad \hfill c= 2,\\
			2 \sum_{i=0}^{n-2} \alpha_i + \alpha_{n-1} + \alpha_n \quad \hfill c = 2n > 2,\\
			2 \sum_{i=0}^{n-1} \alpha_i + \alpha_n \quad \hfill c = 2n+1.
		\end{cases}
	\end{equation*}
	For odd $c$ or $c=0$, $C^0_G = C_G$, the center of $G$; for even $c$, $C^0_G$ is the connected component of $C_G$. The Weyl group $W(G)$ of $G$ is canonically isomorphic with $W(\SO_c)$ and given a permutation matrix $w_0 \in \SO_c$ the preimage of $w_0$ in Spin$_c$ consists of 2 elements which differ by an element in $C_{\text{Spin}_c}$. We identify $w \in H = \GS_{2c}$ with each $w_0$ by making a choice of such a representative in $G$. 
	The unipotent subgroups of $H$ are isomorphic to those of $\SO_{2c}$.
	For a definition of general spin groups using the root datum refer to \cite{Asg02, AS06, HS16}.
	
	There is a natural homomorphism from GSpin$_c$ to SO$_c$ sending $g \in \text{GSpin}_c$ to the map $v \mapsto gvg^{-1}$, as given by Deligne in \cite{Del}. The kernel of this map is $F^\times$, giving us the short exact sequence of algebraic groups
	\begin{equation}
		1 \rightarrow \GL_1 \rightarrow \text{GSpin}_c \xrightarrow{\text{proj}} \text{SO}_c \rightarrow 1. \label{SES}
	\end{equation}
	
	\section{The doubling method for the general spin group}
	
	\subsection{The embedding}
	We can define an embedding of $\SO_c \times \SO_c$ into $\SO_{2c}$ and extend it to $\GS_{2c}$ by using the definition above. For $(g_1,g_2) \in \SO_c \times \SO_c$ and $c =2n$, define 
	\begin{equation*}
		i_0(g_1,g_2) = \begin{bmatrix}
			g_{1,1} &  & g_{1,2}\\
			& g_2 &   \\
			g_{1,3} &   & g_{1,4}       
		\end{bmatrix}
	\end{equation*}
	where $g_1 = \begin{bsmallmatrix}
		g_{1,1} &  g_{1,2}\\
		g_{1,3}& g_{1,4}  
	\end{bsmallmatrix}$, $g_{1,i} \in \text{Mat}_c$, the $c \times c$ matrices over $F$. For $c = 2n+1$, take the column vectors $e_{\pm i}$, $1 \leq i \leq 2n+1$ with Gram matrix $J_{2c}$. Let
	\begin{align*}
		&b = (e_1,\ldots,e_{2n}, (1/2) e_{2n+1}-e_{-2n-1},(1/2)e_{2n+1}+e_{-2n-1},e_{-2n},\ldots,e_{-1}),\\
		&b_1 = (e_1,\ldots,e_n, (1/2)e_{2n+1}-e_{-2n-1},e_{-n},\ldots,e_{-1})\\
		&b_2 = (e_{n+1},\ldots,e_{2n},(1/2)e_{2n+1}+e_{-2n-1},e_{-2n},\ldots,e_{-n-1})\\
		&m = \text{diag}\left(I_{c-1},\begin{bsmallmatrix}
			\frac{1}{2} & \frac{1}{2}\\
			-1 & 1
		\end{bsmallmatrix}, I_{c-1}\right).
	\end{align*}
	The Gram matrices of $(b,b_1,b_2)$ are $(J_{2c}, \text{diag}(I_n,-1,I_n)J_c, J_c)$. We define the left copy of $\SO_{c}$ using $b_1$, i.e. group of matrices $g_1 \in \SL_{c}$ such that 
	\begin{equation*}
		\prescript{t}{}{g_1} \, \text{diag}(I_n,-1,I_n)J_c \,g_1 = \text{diag}(I_n,-1,I_n)J_c.
	\end{equation*}
	The right copy is defined using $b_2$ and our standard convention. For each $i$, extend $g_i$ by
	letting it fix the vectors of $b_{3-i}$, then write this extension as a matrix $g'_i \in \SO_{2c}$ with respect to $b$. Now $\prescript{m}{}g'_1 (= m g'_1m^{-1})$ and $\prescript{m}{}g'_2$ commute and 
	\begin{equation*}
		i_0(g_1,g_2) = [\prescript{m}{}g'_1\prescript{m}{}g'_2]
	\end{equation*}
	
	We extend the mapping to $i : G\times G \rightarrow H$ such that it is injective on the product of derived groups via the embedding of \cite{CFK}. The two commuting copies of $G$ intersect non-trivially and the map factors through $\faktor{G\times G}{(C^0_G)^\Delta}  \rightarrow H$. 
	Starting with the derived groups, the embedding described above for the orthogonal groups
	extends to an embedding of the direct product of derived groups, since it identifies each root
	subgroup of a copy of $G$ with a unipotent subgroup in $H$. The first coordinate
	map of the left copy is identified with $-\epsilon_0$, and the right copy with $\epsilon_0$. 
	
	For $z \in C^0_G$, we have $(z,1) = (1,z^{-1}) \in C_H^0$ and 
	\begin{equation*}
		(G,1) \cap (1,G)  = (C^0_G,1) \cap (1,C^0_G) = C^0_H
	\end{equation*}
	with $(z,z)$ being the identity element. For the right copy, let $g \mapsto \prescript{\iota}{}{g} (= \iota g \iota^{-1})$ denote the outer involution of $g$ which is the canonical extension of the involution in \cite[Section 3.4]{CFK}. 
	For further details, refer to \cite{CFK}.

	\subsection{The integral}
	Let $P$ be the standard maximal parabolic subgroup of $H$ with the Levi part $M_P = \GL_c \times \GL_1$. For quasi-characters $\chi,\eta$ of $F^\times$, let $\chi \times \eta \coloneqq ( \chi \circ \det) \otimes \eta$ be a character of $M_P$. For a complex parameter $s$, let $V(s, \chi \times \eta)$ be the space of $\ind_P^H((\chi \circ \det)|\det|^{s-1/2} \otimes \eta)$. The induction is normalized. Let $\pi$ be an irreducible representation of $G$, acting in a space $V_\pi$. Consider
	the contragredient representation $\widehat{\pi}$ acting in $\widehat{V}_\pi$, the smooth dual of $V_\pi$ and denote by $\langle, \rangle$ the canonical $G$-invariant bilinear form on $V_\pi \times \widehat{V}_\pi$. Let $v_1 \in V_\pi$, $\widehat{v}_2  \in \widehat{V}_\pi$, and let $f_{s} \in V(s, \chi \times \eta)$ be a holomorphic section for $\eta = \pi\vert_{C^0_G}$. The local zeta integrals attached to $\pi \times \chi$ by the doubling method are
	\begin{equation*}
		Z(v_1, \widehat{v}_2, f_s) = \int_{G/C_G^0} \langle v_1, \pi(g)\widehat{v}_2 \rangle f_{s}(\delta i(1, \prescript{\iota}{} g) ) dg.
	\end{equation*}
	We use the notation $p = \{p_0,t\}$ to denote $p \in P$ in terms of its projection $p_0 =$ proj$(p)$ and $t \in F^\times$, the $\GL_1$ part of its Levi. Here $\delta = \mathfrak{p} \delta_0 \delta_1$ with
	\begin{align*}
		&\mathfrak{p} = 
		\left\{\begin{bmatrix}
			& -I_n & & \frac{1}{2} I_n \\
			I_n & & \frac{1}{2}I_n & \\
			& & & I_n \\
			& & -I_n &
		\end{bmatrix},1\right\} \in P, \\
		&\text{ proj}(\delta_0) = 
		\begin{bmatrix}
			& I_{c} \\
			I_{c} & 
		\end{bmatrix},
		\text{ proj}(\delta_1) = 
		\begin{bmatrix}
			I_{c} & A \\
			& I_{c}
		\end{bmatrix} 
	\end{align*}
	where proj$(g)$ is the projection of $g$ onto $\SO_{2c}$ and 
	$A = 
	\begin{bsmallmatrix}
		-I_n & \\
		& I_n
	\end{bsmallmatrix}$ for $c = 2n$. For $c =2n+1$, we have
	\begin{align*}
		&\mathfrak{p} = 
		\left\{\begin{bmatrix}
			& & I_n & & & \frac{1}{2} I_n \\
			& 1 & & & 0 & \\
			-I_n & & & \frac{1}{2} I_n & &\\
			& & & & & -I_n \\
			& & & & 1 & \\
			& & & I_n & & 
		\end{bmatrix},1\right\} \in P,\\
		&\text{ proj}(\delta_0) = 
		\begin{bmatrix}
			& I_c \\
			I_c & 
		\end{bmatrix} 
		\text{diag}
		\left(\begin{bmatrix}
			I_n & & \\
			& & -1 \\
			& I_n &
		\end{bmatrix},
		\begin{bmatrix}
			& I_n &\\
			-1 & & \\
			& & I_n
		\end{bmatrix}\right)j_c,\\
		&\text{ proj}(\delta_1) = 
		\begin{bmatrix}
			I_{c} & A \\
			& I_{c}
		\end{bmatrix} 
	\end{align*}
	with $A = \begin{bsmallmatrix}
		& -I_n & \\
		& & I_n \\
		0 & & 
	\end{bsmallmatrix}$. Here \begin{equation*}
		j_c = \begin{cases}
			\begin{bsmallmatrix}
				I_{2n} & & & \\
				& & 1 & \\
				& 1 & & \\
				& & & I_{2n}
			\end{bsmallmatrix} \quad & c = 2n+1 \\
			I_{2c} & c = 2n.
		\end{cases}
	\end{equation*} 
	Note, our $\delta$ differs from $\delta = \delta_0\delta_1$ chosen in \cite{CFK} by $\mathfrak{p} \in P$. However, our $Z(v_1, \widehat{v}_2, f_s)$ is identical to the local zeta function $Z'(v_1, \widehat{v}_2, f_s)$ in \cite{CFK} since
	\begin{align*}
		Z(v_1, \widehat{v}_2, f_s) =& \int_{G/C_G^0} \langle  v_1, \pi(g)\widehat{v}_2 \rangle f_{s}(\delta i(1, \prescript{\iota}{} g) ) dg \\
		=&\int_{G/C_G^0} \langle  v_1, \pi(g) \widehat{v}_2 \rangle f_{s}(\mathfrak{p}\delta_0\delta_1 i(1, \prescript{\iota}{} g) ) dg\\
		=& \chi(1)\eta(1)|1|^{c(s-1/2)} \int_{G/C_G^0} \langle v_1,  \pi(g) \widehat{v}_2 \rangle f_{s}(\delta_0\delta_1 i(1, \prescript{\iota}{} g) ) dg\\
		=& Z'(v_1, \widehat{v}_2, f_s).
	\end{align*}
	The integral converges absolutely in a right-half plane and continues to a meromorphic function on the whole plane. This function is rational in $q^{-s}$. We keep denoting the meromorphic continuation by $Z(v_1, \widehat{v}_2, f_s)$. 
	
	Consider the intertwining operator 
	\begin{equation*}
		M(\chi \times \eta, s): V(s, \chi \times \eta) \rightarrow V(1-s, \chi'\times \eta)
	\end{equation*}
	with $\chi'\times \eta = (\eta^{-1}\chi^{-1}\circ \det) \otimes \eta$. The operator is defined for $\Re(s) \gg 0$ by 
	\begin{equation*}
		M(\chi \times \eta, s)f_{s}( h) = \int_{U_{P'}}f_s(w_p^{-1} u h)du,
	\end{equation*}
	where $U_{P'}$ is the unipotent radical of $P'$ which can be identified with the lift of $j_c P_0 j_c$. Note, $U_P \cong \{\begin{bsmallmatrix}
		1 & x\\
		& 1
	\end{bsmallmatrix}
	\big\vert
	x \in \text{Mat}_c(F); \, \prescript{\mathrm t}{} x J + Jx = 0\}$, the unipotent subgroup of $\SO_{2c}$ and $U_{P'} \cong j_c U_P j_c$. The Weyl element $w_p$ can be identified with 
	$j_c
	\begin{bsmallmatrix}
		& I_c \\
		I_c & \\
	\end{bsmallmatrix}(-I_{2c})$. For sake of notational convenience we let $w = \begin{bsmallmatrix}
		& I_c \\
		I_c & \\
	\end{bsmallmatrix}$. The operator is defined on the whole plane by the meromorphic continuation. 
	
	We can define a second zeta integral on the image of $M(\chi \times \eta, s)f_{s}( h)$ from $P'$,  in which case the $\delta$ chosen in the case of $c= 2n+1$ is different. We shall use $\delta' = \mathfrak{p} \delta_0 \delta_1$ with 
	\begin{align*}
		&\mathfrak{p} = 
		\left\{j_c\begin{bmatrix}
			& & -I_n & & & \frac{1}{2} I_n \\
			& -1 & & & 0 & \\
			I_n & & & \frac{1}{2} I_n & &\\
			& & & & & I_n \\
			& & & & -1 & \\
			& & & -I_n & & 
		\end{bmatrix}j_c,1\right\} \in P,\\
		&\text{ proj}(\delta_0) = j_c
		\begin{bmatrix}
			& I_c \\
			I_c & 
		\end{bmatrix} 
		\text{diag}
		\left(\begin{bmatrix}
			I_n & & \\
			& & 2 \\
			& I_n &
		\end{bmatrix},
		\begin{bmatrix}
			& I_n &\\
			\frac{1}{2} & & \\
			& & I_n
		\end{bmatrix}\right),\\
		&\text{ proj}(\delta_1) = 
		\begin{bmatrix}
			I_{c} & A \\
			& I_{c}
		\end{bmatrix}
		\begin{bmatrix}
			I_n & & & & \\
			& -I_n & & & \\
			& & I_2 & & \\
			& & & -I_n & \\
			& & & & I_n
		\end{bmatrix}.
	\end{align*}
	By moving the Levi part of $\mathfrak{p}$ across, we can show 
	\begin{equation*}
		Z(v_1, \widehat{v}_2, M^*(\chi \times \eta, s)f_{s}) = \chi((-\varepsilon)^c)\eta((-\varepsilon)^c)Z'(v_1, \widehat{v}_2, M^*(\chi \times \eta, s)f_{s}),
	\end{equation*}
	where  
	\begin{equation*}
		\varepsilon = \begin{cases}
			1 & c = 2n+1 \\
			-1 & c = 2n.
		\end{cases}
	\end{equation*} 
	and $M^*(s, \chi \times \eta, \psi)f_{s}$ is $M(\chi \times \eta, s)f_{s}$ normalized by the factor $C(s,c,\chi \times \eta, \psi)$ defined in \cite[Section 3]{CFK}.
	Let \begin{equation*}
		\delta_\varepsilon = \begin{cases}
			\delta' & c = 2n+1 \\
			\delta & c = 2n.
		\end{cases}
	\end{equation*} 
	
	The two zeta integrals are related by the functional equation
	\begin{align*}
		\Gamma(s, \pi \times \chi, \psi) =& \frac{Z(v_1, \widehat{v}_2, M^*(\chi \times \eta, s)f_{s})}{Z(v_1, \widehat{v}_2, f_s)} \\
		=& \frac{\chi((-\varepsilon)^c)\eta((-\varepsilon)^c) Z'(v_1, \widehat{v}_2, M^*(\chi \times \eta, s)f_{s})}{Z'(v_1, \widehat{v}_2, f_s)}\\
		=& \chi((-\varepsilon)^c)\eta((-\varepsilon)^c)\gamma(s, \pi \times \chi, \psi) (\pi(i_G)\vartheta(s,c,\chi\otimes\eta,\psi))^{-1}  \numberthis \label{gamma}
	\end{align*} 
	for all $v_1 \in V_\pi$, $\widehat{v}_2  \in \widehat{V}_\pi$, and $f_{s} \in V(s, \chi \times \eta)$. The local $\gamma$-factor $\gamma(s, \pi \times\chi, \psi)$, the constant $i_G$, and the factor $\vartheta(s,c,\chi\otimes\eta,\psi)$ are as defined in \cite[Section 5.1]{CFK}. The function $	\Gamma(s, \pi \times \chi, \psi)$ depends on the choice of the measure $du$, made in the definition of $M(\chi \times \eta, s)$. The constants relating $\Gamma(s, \pi \times \chi, \psi)$ and $\gamma(s, \pi \times \chi, \psi)$ in \eqref{gamma} depend on the restriction of the central character of $\pi\vert_{C^0_G}$ and are otherwise independent of $\pi$. Hence, it is enough to show the stability of $\Gamma(s, \pi \times \chi, \psi)$.
	
	Let the conductors of $\chi$ and $\eta$ be $1+\mathcal{P}^{N_\chi}$ and $1+\mathcal{P}^{N_\eta}$ respectively and put $N_\tau = \max{(N_\chi, N_\eta)}$, $N_\tau \geq 1$ with
	\begin{equation*}
		n_\tau = \left\lfloor\frac{N_\tau +1}{2}\right\rfloor.
	\end{equation*}
	For $l \in \Z$, let 
	\begin{equation}
		\mathfrak{g}_l = \mathfrak{g}(\mathcal{P}_F^l) = \mathfrak{g}\cap \text{Mat}_c(\mathcal{P}_F^l). \label{3.9}
	\end{equation}
	We can now state our main result:
	\begin{theorem} \label{thm3.1}
		Let $\pi$ be an irreducible representation of $G$. There exists a positive integer $N$ such that, for any ramified character $\chi$ and $\eta = \pi\vert_{C^0_G}$ with conductors $1+\mathcal{P}^{N_\chi}$ and $1+\mathcal{P}^{N_\eta}$ respectively, and $N_\chi, N_\eta > N$,
		\begin{align*}
			\Gamma(s, \pi \times \chi, \psi)=& \, |2|^{c(s-1/2)}\chi((2\varepsilon)^c)\eta^{1/2}((2\varepsilon)^c)\chi((-\varepsilon)^c)\eta((-\varepsilon)^c) C(s,c,\chi \times \eta, \psi)\\ 
			&\times \int\limits_{\mathfrak{g}(2^{-1}q^{N_\tau-n_\tau-1})} \chi^{-1}(\det(I_c -v)) 
			\eta^{-1/2}(\det(I_c-v)) dv. 
		\end{align*} 
	\end{theorem}
	As a corollary, we obtain the following result which implies Theorem \ref{thm:1.1} (See \cite[Theorem 2]{RS}).
	\begin{cor}
		Let $\pi$ be an irreducible representation of $G$. Then $\gamma(\pi,\chi\times\eta,s)$ is stable for all sufficiently ramified $\chi$, $\eta$. More precisely, there is a positive integer $N$, such that for all ramified quasi-characters $\chi$, $\eta$ of $F^\times$ with conductor having an exponent larger than $N$,
		\begin{equation*}
			\Gamma(s, \pi \times \chi, \psi) = 	M^*(\chi \times \eta,s)f_s(\delta_\varepsilon i(1,\varepsilon)).
		\end{equation*}
		for a certain choice of $f_s$.
	\end{cor}
	
	\section{Proof of Theorem 3.1}
	Choose $f = f_s \in V(s, \chi \times \eta)$ such that it is supported in the open orbit $P \cdot \delta i(G \times G) \delta^{-1} = P \cdot \delta i(1 \times G) \delta^{-1}$, 
	and so that the restriction $f\vert_{\delta i(1\times G) \delta^{-1}}$, regarded as a function of $G$, is the characteristic function of a compact open subgroup modulo $C_G^0$. 
	In more detail, identify a compact open subgroup with a subgroup $U_0$ of $\SO_c$ and let $U$ be subgroup of $G$ whose quotient modulo $C^0_G$ is identified with $U_0$. 
	Then, $U \cong C^0_G U_0$ and for $g = (z,u) \in U$, define $\phi_U(g) \coloneqq \eta(z)\phi_{U_0}(u)$. 
	This is well defined for $U_0$ small enough as $\eta(u)=1$ for $u \in C_G^0 \cap U_0$. 
	Choose $f \in V(s, \chi \times \eta)$ such that $f\vert_{\delta i(1\times G) \delta^{-1}}$, regarded as a function of $G$, vanishes outside $U$ and $f(\delta i(1, \prescript{\iota}{}u) \delta^{-1}) = \phi_U(u)$. Note, we are using $C_G^0 \cong C_H^0$ to define $\eta(z)$ for $z \in C^0_G$, with abuse of notation. Also, for a compact open subgroup U, $\prescript{\iota}{} g \in U \iff g \in U$.
	Define $f_1 = \delta^{-1}\cdot f$, hence $f_1(h) = f(h \delta^{-1})$. 
	
	We assume that $U$ is small enough so that $V_\pi ^U \neq 0$ and let $v_1 \in V_\pi ^U$ be a non-zero element. Choose $\widehat{v}_2 \in \widehat{V}_\pi$ such that $\langle v_1, \widehat{v}_2 \rangle$ = 1. Then the local zeta integral $Z(v_1, \widehat{v}_2, f_1)$ converges for all $s$,  
	\begin{align*}
		Z(v_1, \widehat{v}_2, f_1) &= \int_{G/C_G^0} \langle v_1, \pi(g) \widehat{v}_2 \rangle f_1(\delta i(1,\prescript{\iota}{}g) ) dg\\
		& = \int_{G/C_G^0}\langle  v_1, \pi(g)\widehat{v}_2 \rangle f(\delta i(1,\prescript{\iota}{}g) \delta^{-1} ) dg\\
		&= \int_{U/C_G^0}\langle v_1, \pi(u) \widehat{v}_2 \rangle f(\delta i(1,\prescript{\iota}{}u) \delta^{-1} ) du\\
		& = \int_{U_0}\langle v_1, \widehat{v}_2 \rangle \phi_{U_0}(u) du\\
		& = m(U_0)\langle v_1, \widehat{v}_2 \rangle\\
		& = m(U_0), 
	\end{align*}
	where $m(U_0)$ is the measure of $U_0$. Hence, from the functional equation, we get 
	\begin{equation}
		\Gamma(s, \pi \times \chi, \psi) = \frac{1}{m(U_0)}\int_{G/C_G^0} \langle v_1, \pi(g)\widehat{v}_2 \rangle M^*(\chi \times \eta,s)f_1(\delta_\varepsilon i(1,\prescript{\iota}{}g))dg \label{4.2}
	\end{equation} for $\Re(s) \ll 0$. By meromorphic continuation, this equality is valid for all $s$. Let $1_G$ denote the identity in $G$. We will now show that $M(\chi \times \eta,s)f_1(\delta_\varepsilon i(1_G,\prescript{\iota}{}g))$ is supported in $i(1_G \times \varepsilon U)$ where it is a constant. We will use this to prove Theorem \ref{thm3.1}.
	
	\begin{lemma}
		For $f_1$ as chosen above, we have, for $\Re(s) \gg 0$,
		\begin{align*}
			M(\chi \times \eta,s)f_1(\delta_\varepsilon i(1_G,\prescript{\iota}{}g)) = |2|^{c/2(c-1)}&\chi(\varepsilon^c)\eta^{1/2}(\varepsilon^c)\chi((-\varepsilon)^c)\eta((-\varepsilon)^c)\\ \times  \int\limits_{\substack{h \in \SO_c \\ \det(I_c + h) \neq 0}} &\chi(\det(I_c+h))\eta^{1/2}(\det(I_c+h)\zeta_g)|\det(I_c+h)|^{s-c/2}\phi_{U_0}(\varepsilon hg_0)dh, \numberthis \label{Lem4.1}
		\end{align*}
		where $dh$ is an explicit Haar measure on $\SO_c$ and 
		\begin{align*}
			\zeta_g =
			\begin{cases}
				1 \quad \text{$\Upsilon(\delta_\varepsilon i(1,\inv g) \delta^{-1})$ is a square},\\
				\zeta \quad \text{otherwise},
			\end{cases}
		\end{align*} for a chosen non-square element $\zeta \in F$. 
	\end{lemma}
	We shall later show that $\zeta_g = 1$ for all relevant $g$.
	\begin{proof}
		Let $\Re(s) \gg 0$, so that the integral converges absolutely. By the definition of the intertwining operator, we have 
		\begin{align*}
			M(\chi\times \eta,s)f_1(\delta i(1_G, \prescript{\iota}{}g)) &= \int_{U_{P'}} f_1((-1_{2c})w_P^{-1}u \delta_\varepsilon i(1,\inv g)) du \\
			& = \int_{U_P} f((-1_{2c})wu j_c^\# \delta_\varepsilon i(1,\inv g)\delta^{-1}) du
		\end{align*}
		where $j_c^\#$ is the extension of $j_c$ to $\GS_{2c}$ and $-1_{2c} \in \GS_{2c}$ is the element $\{-I_{2c},1\}$. Note, $j_c$ and $j_c^\#$ are trivial if $c = 2n$.
		By our choice of $f$, the value is non-zero only if 
		\begin{equation*}
			w \cdot u \cdot j_c^\# \cdot \delta_\varepsilon i(1, \inv g) \delta^{-1} \in P \cdot \delta i(1_G \times G) \delta^{-1}
		\end{equation*} 
		i.e.,
		\begin{equation*}
			w \cdot u \cdot j_c^\# \in P \cdot \delta i(1_G \times G) \delta\varepsilon^{-1}. 
		\end{equation*}
		More explicitly, there exists $p \in P$ and $h \in G$ such that 
		\begin{equation}
			w \cdot u \cdot j_c^\# = p \cdot \delta i(1,\inv h) \delta_\varepsilon^{-1}  \label{wu}
		\end{equation}
		Let $w_0, u_0, h_0$ and $p_0$ denote the images of $w,u,h$ and $p$ respectively under the map \eqref{SES}. Let $\delta_0$ and $\delta_{\varepsilon,0}$ henceforth denote the projections of $\delta$ and $\delta_\varepsilon$ respectively. (We shall no longer be using $\delta_0$ as previously defined in Section 3.2, so there should be no confusion.)
		\begin{claim}
			$ w_0 \cdot u_0 \cdot j_c = p_0 \cdot \delta_0 i_0(1, \inv h_0) \delta_{\varepsilon,0}^{-1}$ has a solution in $\SO_{2c}$ if and only if  
			$w\cdot u \cdot j_c^\# = p \cdot \delta i(1, \inv h) \delta_\varepsilon^{-1}$ respectively has a solution in $\GS_{2c}$.  \label{claim1}
		\end{claim}
		\begin{proof}
			Solution in $\GS_{2c}$ clearly implies a solution in $\SO_{2c}$. To show the converse, we extend a solution in $\SO_{2c}$ to a solution in $\GS_{2c}$ according to the construction of \cite{CFK}. This extension is up to the choice of preimage and is unique up to an element of the center. 
		\end{proof}
		
		Writing $u_0 = \begin{bmatrix}
			I & x \\
			& I
		\end{bmatrix}$ for $x \in \text{Mat}_c(F)$ such that $\prescript{\mathrm t}{}x J + Jx = 0$ and $I = I_c$, the condition in Claim \ref{claim1} is equivalent to solving for such an $x$ in 
		\begin{equation*}
			\begin{bmatrix}
				& I\\
				I & \\
			\end{bmatrix}
			\begin{bmatrix}
				I & x \\
				& I
			\end{bmatrix} =
			\begin{bmatrix}
				E & Y \\
				& E^*
			\end{bmatrix}
			\begin{bmatrix}
				\frac{1}{2}(h_0-\varepsilon I) & \frac{1}{4} (h_0+\varepsilon I) \\
				(h_0 +\varepsilon I) & \frac{1}{2} (h_0 -\varepsilon I)
			\end{bmatrix}.
		\end{equation*}
		Here $E^* = J \prescript{\mathrm t}{}{E^{-1}} J$, and $Y$ is such that $\begin{bmatrix}
			E & Y \\
			& E^*
		\end{bmatrix} \in P_0$. A solution for this system exists if and only if 
		\begin{equation*}
			\det (I - 2 x) \neq 0,
		\end{equation*}
		in which case, we have 
		\begin{equation}
			h_0 =\varepsilon \frac{I +2x}{I -2x} \label{4.8}
		\end{equation}
		and 
		\begin{equation}
			E = (2\varepsilon)(I+2x)^{-1}.
		\end{equation}
		Note that since $-x = J\prescript{t}{}{x}J$,
		\begin{equation}
			\det(I+2x) = \det(I-2x),
		\end{equation}
		and hence, $\det(h_0) = 1$.
		
		Consider the Cayley transform for elements of $\SO_c$, as in \cite{RS},
		$$\mathfrak{c} : \mathfrak{g}' \rightarrow G^{0'}  $$
		$$\mathfrak{c}(y) = \frac{I+y}{I-y}$$
		where $$\mathfrak{g}' = \{y \in \mathfrak{so}_c| \det(I-y) \neq 0 \},$$
		$$G^{0'}  = \{t \in \SO_c| \det(I+t) \neq 0\}.$$
		Hence, \eqref{4.8} implies 
		\begin{equation}
			h_0 = \varepsilon\mathfrak{c}(2x)
		\end{equation}
		and 
		\begin{equation}
			\{\prescript{\mathrm t}{}x J + Jx = 0 \text{ and } \det (I - 2 x) \neq 0 \} \Leftrightarrow x \in \mathfrak{g}', \label{4.13}
		\end{equation}
		by \eqref{2.3}.
		
		Note, the uniqueness of $p_0$ and $h_0$ only determines $p$ and $h$ in \eqref{wu} up to an element of the center. Using the notation as in Section 3.2, we denote 
		$$p = \left\{\begin{bmatrix}
			E & Y\\
			& E^*
		\end{bmatrix}, t\right\}.$$ To make a canonical choice of $p$ and $h$ we claim the following:
		\begin{claim}
			If $g\in U$, or if $-g\in U$ and $c =2n$, then $\Upsilon(\delta_\varepsilon i(1,\inv g) \delta^{-1})$ is a square. \label{claim2}
		\end{claim}
		\begin{proof}
			By the Iwahori factorization in $G$, 
			\begin{equation*}
				\delta_\varepsilon i(1, U) \delta^{-1}  = (\delta_\varepsilon i(1,U) \delta_\varepsilon^{-1}\cap U_P^-) (\delta_\varepsilon i(1,U) \delta^{-1} \cap P)
			\end{equation*} 
			and hence,
			\begin{equation*}
				\delta_{\varepsilon,0} i(1, U_0) \delta_0^{-1} = (\delta_{\varepsilon,0} i(1, U_0) \delta_{\varepsilon,0}^{-1} \cap U_P^-) \left(\delta_{\varepsilon,0} i(1, U_0) \delta_0^{-1} \cap \begin{bsmallmatrix}
					E_u & Y_u \\
					& E_u^*
				\end{bsmallmatrix}\right)
			\end{equation*}
			as unipotents of $\SO_{2c}$ and $\GS_{2c}$ are canonically isomorphic. Now,
			\begin{equation*}\delta_{\varepsilon,0} i(1, U_0) \delta_0^{-1} \cap \begin{bmatrix}
					E_u & Y_u \\
					& E_u^*
				\end{bmatrix} = \begin{bmatrix}
					I_c + \text{Mat}_c(\mathcal{P}^N) & \text{Mat}_c(\mathcal{P}^N)\\
					& I_c + \text{Mat}_c(\mathcal{P}^N)
				\end{bmatrix}
			\end{equation*}
			and hence $\det (E_u) \in 1 + \mathcal{P}^N$. Note,  
			\begin{equation}
				\Upsilon\vert_{\delta_\varepsilon i(1,U) \delta_\varepsilon^{-1}\cap U_P^-} = 1 \label{4.17}
			\end{equation}
			as it is a character. Denoting an element $u \in \delta_\varepsilon i(1,U) \delta^{-1} \cap P$ by $$u = \left\{\begin{bmatrix}
				E_u & Y_u\\
				& E_u^*
			\end{bmatrix}, t_u\right\},$$ we get, by \eqref{4.17} and the definition of $\Upsilon$ (c.f. \cite[Section 3]{CFK}) 
			\begin{equation}
				\Upsilon(u) = \det (E_u) t_u ^{-2}. \label{4.10}
			\end{equation}
			which is a square for any $N > c_F$, a constant depending only on $F$. Hence, if $g \in U$ then $\Upsilon(\delta_\varepsilon i(1,\inv g) \delta^{-1}) = \det(E_g)t_g^{-2}$ is a square. If $-g \in U$ then $\Upsilon(\delta_\varepsilon i(1,\inv g) \delta^{-1}) = \det(-E_{-g})t_{-g}^{-2} = (-1)^c \det(E_{-g})t_{-g}^{-2}$ all of which are squares for even $c$. 
		\end{proof}
		
		Now, by \eqref{wu}
		\begin{equation*}
			f((-1_{2c})wuj_c^\# \cdot \delta_\varepsilon i(1,\inv g) \delta^{-1}) = f((-1_{2c})p \cdot \delta i(1,\inv (hg)) \delta^{-1}) = (\chi \times \eta) (-1_{2c}p) f(\delta i(1,\inv (hg)) \delta^{-1})
		\end{equation*} is non-zero only if $hg \in U$. Hence, 
		\begin{equation*}
			\Upsilon(\delta i(1,\inv h) \delta_\varepsilon^{-1})\Upsilon(\delta_\varepsilon i(1,\inv g) \delta^{-1}) = \Upsilon(\delta i(1,\inv (hg)) \delta^{-1}) =  \det(E_{hg})t_{hg}^{-2},
		\end{equation*} where $E_{hg}$ and $t_{hg}$ are defined similarly to those in proof of Claim \ref{claim2}. 
		Therefore, \begin{equation}
			\Upsilon(\delta i(1,\inv h) \delta_\varepsilon^{-1}) = \Upsilon(\delta_\varepsilon i(1,\inv g) \delta^{-1})^{-1}\det(E_{hg})t_{hg}^{-2}. \label{h}
		\end{equation} 
		Note that changing $h \mapsto zh$ for any element $z = \{I_c,t_z\}$ of $C^0_G$ changes $\Upsilon(h) \mapsto \Upsilon(h)t_z^{-2}$. 
		
		The term $\det(E_{hg})t_{hg}^{-2}$ is a square by reasons of \eqref{4.10}. Therefore, \eqref{h} implies $\Upsilon(\delta i(1,\inv h) \delta_\varepsilon^{-1})$ is a square if we can further show that $\Upsilon(\delta_\varepsilon i(1,\inv g) \delta^{-1})$ is a square, for which it suffices to show either $g$ belongs to $U$ or both $-g$ belongs to $U$ and $c =2n$. If $\Upsilon(\delta_\varepsilon i(1,\inv g) \delta^{-1})$ is a square, we make the canonical choice of $h$ over $h_0$ such that
		\begin{equation}
			t_{hg}= (\Upsilon(\delta_\varepsilon i(1,\inv g) \delta^{-1})^{-1}\det(E_{hg}))^{1/2}	
		\end{equation} or equivalently, $\Upsilon(\delta i(1,\inv h) \delta_\varepsilon^{-1})) = 1$. Otherwise, choose
		$h$ such that 
		\begin{equation}
			\Upsilon(\delta_\varepsilon i(1,\inv g) \delta^{-1})^{-1}\det(E_{hg})t_{hg}^{-2} = \zeta
		\end{equation}
		for a fixed non-square element of $F$. We denote 
		\begin{equation}
			\Upsilon(\delta i(1,\inv h) \delta_\varepsilon^{-1}) = \zeta_g
		\end{equation} where 
		\begin{align*}
			\zeta_g =
			\begin{cases}
				1 \quad \text{$\Upsilon(\delta_\varepsilon i(1,\inv g) \delta^{-1})$ is a square},\\
				\zeta \quad \text{otherwise},
			\end{cases}
		\end{align*}
		which is independent of $h$ (and hence $x$) and depends only on $g$. 
		
		Now, $\Upsilon(p \cdot \delta i(1,\inv h) \delta_\varepsilon^{-1})= \Upsilon(w u j_c^\#) =1$ and therefore, 
		\begin{equation*}
			1 = \Upsilon(p)\Upsilon(\delta i(1,\inv h) \delta_\varepsilon^{-1}) = \det(E)t^{-2}\Upsilon(\delta i(1,\inv h) \delta_\varepsilon^{-1}) = \det(E)t^{-2}\zeta_g.
		\end{equation*} Therefore,
		\begin{equation}
			t = (\det(E) \zeta_g)^{1/2}. \label{4.25}
		\end{equation}
		Note, the choice of $t$ between the two branches of the square root does not matter, as choosing $-t$ instead of $t$ changes both $p$ and $h$ preserving \eqref{wu}. Now $\{I,-1\} \in C^0_G$ and for all $z \in C^0_G$,
		\begin{align*}
			f(zp \cdot \delta i(1,\inv z^{-1} h) \delta_\varepsilon^{-1}) &= (\chi\times\eta)(zp) \phi_U(z^{-1}\delta i(1,\inv h) \delta_\varepsilon^{-1})\\
			&= \eta(z)(\chi\times\eta)(p)\eta(z^{-1})\phi_U(\delta i(1,\inv h) \delta_\varepsilon^{-1})\\
			& = (\chi\times\eta)(p)f(\delta i(1,\inv h) \delta_\varepsilon^{-1})\\
			& = f(p\cdot \delta i(1,\inv h) \delta_\varepsilon^{-1})
		\end{align*}
		(here we use $C^0_G \cong C^0_H$ and denote both elements by $z^{-1}$). 
		
		For ease of notation, let $\chi((-\varepsilon)^c)\eta((-\varepsilon)^c)Mf(g) = M(\chi \times \eta,s)f_1(\delta_\varepsilon i(1,\inv g))$. Hence, by \eqref{4.8}--\eqref{4.25}, for $\Re(s) \gg 0$, we get 
		\begin{align*}
			Mf(g)=& \int\limits_{\substack{J\prescript{\mathrm t}{}x+xJ=0 \\ \det(I -2x) \neq 0}} \chi(\det(E)) \eta^{1/2}(\det(E)\zeta_g) | \det(E)|^{s+c/2-1} \phi_{U_0}(h_0g_0)dx\\
			=& \,|2|^{ c(s+c/2-1)}\chi((2\varepsilon)^c)\eta^{1/2}((2\varepsilon)^c) \\
			&\times \int\limits_{\substack{J\prescript{\mathrm t}{}x+xJ=0 \\ \det(I -2x) \neq 0}} \chi^{-1}(\det(I -2x)) \eta^{-1/2}(\det(I-2x)\zeta_g^{-1})|(\det(I-2x))|^{-(s+c/2-1)}\phi_{U_0}(\varepsilon\mathfrak{c}(2x)g_0)dx. 
		\end{align*}
		By change of variables $x \mapsto \frac{1}{2}x$, we can write
		\begin{align*}
			Mf(g) =& \, |2|^{\frac{c}{2}(2s-1)}\chi((2\varepsilon)^c)\eta^{1/2}((2\varepsilon)^c) \\ 
			&\times \int\limits_{\substack{J\prescript{\mathrm t}{}x+xJ=0 \\ \det(I -x) \neq 0}}	 \chi^{-1}(\det(I -x)) \eta^{-1/2}(\det(I-x)\zeta_g^{-1}) |(\det(I-x))|^{-(s+c/2-1)}\phi_{U_0}(\varepsilon\mathfrak{c}(x)g_0)dx. 
		\end{align*}
		Using \eqref{4.13}, we can change the domain of integration to $\mathfrak{g}'$ by letting $y=x$ with the measure $dy = dx$. Then,
		\begin{align*}
			Mf(g) =& \, |2|^{\frac{c}{2}(2s-1)}\chi((2\varepsilon)^c)\eta^{1/2}((2\varepsilon)^c) \\ 
			&\times \int_{\mathfrak{g}'} \chi^{-1}(\det(I -y)) \eta^{-1/2}(\det(I-y)\zeta_g^{-1}) |(\det(I-y))|^{-(s+c/2-1)}\phi_{U_0}(\varepsilon\mathfrak{c}(y)g_0)dy. \numberthis \label{4.28}
		\end{align*}
		Now by Lemma 4.2 of \cite{RS} along with the change of variables $\mathfrak{c}(y)=h$, we get
		\begin{align*}
			Mf(g) =& \, |2|^{\frac{c}{2}(2s-1)}\chi((2\varepsilon)^c)\eta^{1/2}((2\varepsilon)^c) \\ 
			& \times \int_{G^{0'}} \chi^{-1}(\det(2(I +h)^{-1})) \eta^{-1/2}(\det(2(I+h)^{-1})\zeta_g^{-1}) |(\det(2(I+h)^{-1}))|^{-(s-c/2)}\phi_{U_0}(\varepsilon hg_0)dh\\
			=& \,|2|^{\frac{c}{2}(c-1)}\chi(\varepsilon^c)\eta^{1/2}(\varepsilon^c) \\ 
			&\times \int_{G^{0'}} \chi^{-1}(\det(I +h)^{-1}) \eta^{-1/2}(\det(I+h)^{-1}\zeta_g^{-1})|(\det(I+h)^{-1})|^{-(s-c/2)}\phi_{U_0}(\varepsilon hg_0)dh\\
			=&\, |2|^{\frac{c}{2}(c-1)}\chi(\varepsilon^c)\eta^{1/2}(\varepsilon^c) \\ 
			&\times \int_{G^{0'}} \chi(\det(I +h)) \eta^{1/2}(\det(I+h)\zeta_g)|\det(I+h)|^{(s-c/2)}\phi_{U_0}(\varepsilon hg_0)dh. 	
		\end{align*} 
		for a suitable Haar measure $dh$ on $G^{0'}$ as in \cite{RS}. 
	\end{proof}
	
	We note that the following Lemma from \cite{RS} about Cayley transforms remains valid, even though the definitions of $\mathfrak{g}'$ and $G^{0'}$ are different. 
	\begin{lemma} \label{lem4.2}
		Let $v \in \mathfrak{g}'$ and $h \in G^{0'}$. Then
		\begin{equation*}
			I - \mathfrak{c}(v)h = (I - v)^{-1}(-\mathfrak{c}^{-1}(h)-v)(I + h),
		\end{equation*}
		\begin{equation}
			I + \mathfrak{c}(v)h = (I - v)^{-1}(I + v\mathfrak{c}^{-1}h)(I + h). \label{4.29}
		\end{equation}
		Hence $I + \mathfrak{c}(v)h$ is invertible if and only if $I + v\mathfrak{c}^{-1}(h)$ is invertible, in which case
		\begin{equation*}
			\mathfrak{c}^{-1}(\mathfrak{c}(v)h) = (I - v)^{-1}(\mathfrak{c}^{-1}(h)+v)(I+v\mathfrak{c}^{-1}(g))(I -v).
		\end{equation*}
	\end{lemma}
	
	For a matrix $x \in \text{Mat}_c(F)$, let
	$$||x||=||x||_\infty = \max_{1\leq i,j \leq c}|x_{i,j}|.$$
	We have 
	\begin{equation*}
		||k_1 x k_2|| = ||x||, \,\forall x \in \text{Mat}_c(F), \, \forall k_1,k_2 \in \GL_c(\mathcal{O}_F).
	\end{equation*}
	Let $N$ be a positive even integer such that 
	\begin{equation}
		q^N > |4|^{-1}q^4 \label{4.32}
	\end{equation}
	and such that 
	\begin{equation}
		\mathfrak{c}(\mathfrak{g}(2^{-1}\mathcal{P}_F^{\frac{N-2}{2}})) \in U_0, \label{4.33}
	\end{equation}
	defined as in \eqref{3.9}. Let $1 + \mathcal{P}_F^{N_\chi}$ and $1+\mathcal{P}_F^{N_\eta}$ be the conductors of $\chi$ and $\eta$ respectively. We will assume that $N_\chi,N_\eta > N$ and let $N_\tau = \max(N_\chi,N_\eta)$ with $n_\tau = \lfloor\frac{N_\tau +1}{2}\rfloor$.
	
	Replacing $g$ with $\varepsilon g^{-1}$ in integral \eqref{Lem4.1}, we must have $hg_0^{-1} = u \in U_0$ i.e. $h = ug_0$ for some $u \in U_0$. We can then write \eqref{Lem4.1} at $\varepsilon g^{-1}$ as 
	\begin{align*}
		Mf(\varepsilon g^{-1})
		=& \, |2|^{\frac{c}{2}(c-1)}\chi(\varepsilon^c)\eta^{1/2}(\varepsilon^c)\sum_{L=-\infty}^{\infty} \\ 
		&\times \int\limits_{\substack{||\mathfrak{c}^{-1}(ug_0)|| = q^L\\ u \in U_0}} \chi(\det(I +ug_0)) \eta^{1/2}(\det(I+ug_0)\zeta_g) |\det(I+ug_0)|^{(s-c/2)}du. \numberthis	\label{4.34}
	\end{align*}
	
	For $L \in \Z$ and $\Re(s) \gg 0$, denote
	\begin{align*}
		I_L(\chi \times \eta, g) = &\int\limits_{\substack{||\mathfrak{c}^{-1}(ug_0)|| = q^L\\ u \in U_0}} \chi(\det(I +ug_0)) \eta^{1/2}(\det(I+ug_0)\zeta_g) |\det(I+ug_0)|^{(s-c/2)}du. 
	\end{align*}
	
	\begin{lemma} \label{lem4.3}
		For all $L \geq n_\tau$
		\begin{equation*}
			I_L(\chi \times \eta, g) = 0.
		\end{equation*}
	\end{lemma}
	\begin{proof}
		We make a change of variables $ u \mapsto \mathfrak{c}(v)u$ for a fixed $v \in \mathfrak{g}_{L+n_\tau} = \mathfrak{g}(\mathcal{P}_F^{L+n_\tau})$ with $\mathfrak{c}(v)u \in U_0$ by \eqref{4.33}. By Lemma \ref{lem4.2}, we know $I + \mathfrak{c}(v)ug_0$ is invertible if and only if $I + v\mathfrak{c}^{-1}(ug_0)$ is invertible. Furthermore, we have
		\begin{equation}
			||v\mathfrak{c}^{-1}(ug_0)|| \leq q^{-n_\tau} < 1 \label{4.37}
		\end{equation}
		and
		\begin{equation*}
			||\mathfrak{c}^{-1}(\mathfrak{c}(v)ug_0)|| = q^L.
		\end{equation*}
		Let $dv$ be the Haar measure on $\mathfrak{g}$ chosen prior to \eqref{4.28}. Then
		\begin{align*}
			I_L(\chi \times \eta, g) =& \frac{1}{m(\mathfrak{g}_{L+n_\tau})}
			\int\limits_{\mathfrak{g}_{L+n_\tau}} \int\limits_{\substack{||\mathfrak{c}^{-1}(ug_0)|| = q^L\\ u \in U_0}} \chi(\det(I+\mathfrak{c}(v)ug_0))\\
			&\times (\eta(\det(I+\mathfrak{c}(v)ug_0)\zeta_g))^{1/2} |\det(I+\mathfrak{c}(v)ug_0)|^{(s-c/2)}du dv, 
		\end{align*}
		where $m(\mathfrak{g}_{L+n_\tau})$ is the measure of $\mathfrak{g}_{L+n_\tau}$ under $dv$. By \eqref{4.29}, we have
		\begin{align*}
			&\chi(\det(I+\mathfrak{c}(v)ug_0)) (\eta(\det(I+\mathfrak{c}(v)ug_0)\zeta_g))^{1/2} |\det(I+\mathfrak{c}(v)ug_0)|^{(s-c/2)} \\
			&= \chi(\det(I+v\mathfrak{c}^{-1}(ug_0))) \chi(\det(I+ug_0))\\
			&\quad \times (\eta(\det(I+v\mathfrak{c}^{-1}(ug_0))\det(I+ug_0)\zeta_g))^{1/2}|\det(I+ug_o)|^{(s-c/2)}. 
		\end{align*}
		Here, we used that $\det(I-v) \in 1+ \mathcal{P}_F^{L+n_\tau} \subset 1 + \mathcal{P}_F^{2n_\tau} 
		\subset 1 + \mathcal{P}_F^{N_\tau}$, and hence it is a square. Also, as $N_\tau \geq N_\chi, N_\eta$ we have $\chi(\det(I -v)) = 1$, $(\eta(\det(I-v)))^{1/2} =1$ and $|\det(I-v)| = 1$.
		Similarly, by \eqref{4.37}, $|\det(I + v\mathfrak{c}^{-1}(ug_0))| = 1$. Now, using the observations on page 321 of \cite{JS} and \eqref{4.37}, we have
		\begin{equation*}
			\det(I+v\mathfrak{c}^{-1}(ug_0)) \equiv 1 + \text{trace}(v\mathfrak{c}^{-1}(ug_0)) \quad(\text{mod} \mathcal{P}_F^{2n_\tau})
		\end{equation*}
		and
		\begin{equation}
			\text{trace}(v\mathfrak{c}^{-1}(ug_0))\in \mathcal{P}_F^{n_\tau}. \label{4.44}
		\end{equation}
		Note, this implies $\det(I + v\mathfrak{c}^{-1}(ug_0))$ is a square. Hence, $\eta^{1/2}(I + v\mathfrak{c}^{-1}(ug_0))$ is well defined and has a unique solution by \eqref{4.44}. Let $\tau$ denote the character $\chi \times \eta^{1/2}$ defined on $1 + \mathcal{P}_F^{n_\tau}$. Then the conductor of $\tau$ is $1 + \mathcal{P}_F^{N_\tau}$, as defined above. 
		
		The map $x \mapsto 1+x$ defines a group isomorphism 
		$$\mathcal{P}_F^{n_\tau}/\mathcal{P}_F^{N_\tau} \rightarrow 1+ \mathcal{P}_F^{n_\tau}/ 1+\mathcal{P}_F^{N_\tau}.$$
		As $2n_\tau \geq N_\tau$, the map $x \mapsto \tau(1+x)$ defines a character of $\mathcal{P}_F^{n_\tau}/\mathcal{P}_F^{N_\tau}$. Fix a character $\phi_0$ of $F$ whose conductor is $\mathcal{O}_F$. Then, there is $a \in F^\times$ such that $|a| = q^{N_\tau}$ and
		\begin{equation*}
			\tau(1+x) = \phi_0(ax), \qquad \forall x \in \mathcal{P}_F^{n_\tau}.
		\end{equation*}
		Therefore,
		\begin{equation*}
			\tau(\det(I+v\mathfrak{c}^{-1}(ug_0))) = \tau(1 + \text{trace}(v\mathfrak{c}^{-1}(ug_0))) = \phi_0(a \cdot \text{trace}(v\mathfrak{c}^{-1}(ug_0))).
		\end{equation*}
		Hence,
		\begin{align*}
			I_L(\chi \times \eta, g) = \frac{1}{m(\mathfrak{g}_{L+n_\tau})}\int\limits_{\substack{||\mathfrak{c}^{-1}(ug_0)|| = q^L\\ u \in U_0}}& \chi(\det(I+ug_0)) (\eta(\det(I+ug_0))\zeta_g)^{1/2}
			|\det(I+ug)|^{(s-c/2)}\\
			& \times \left( \int\limits_{v \in \mathfrak{g}_{L+n_\tau}}\phi_0(a \cdot \text{trace}(v\mathfrak{c}^{-1}(ug_0))) dv\right)du. \numberthis \label{phi0}
		\end{align*}
		The inner integral of \eqref{phi0} is 
		\begin{equation}
			\int\limits_{\substack{J\prescript{\mathrm t}{}x+xJ=0 \\ x \in \mathfrak{g}(L+n_\tau)}}\phi_0 (a \cdot \text{trace}(x\mathfrak{c}^{-1}(ug_0)))dx.\label{4.46}
		\end{equation}
		For this integral to be non-zero, we must have 
		\begin{equation*}
			||a \mathfrak{c}^{-1}(ug_0)|| \leq |2|^{-1}q^{L+n_\tau+1}
		\end{equation*}
		and hence
		\begin{equation*}
			||\mathfrak{c}^{-1}(ug_0)|| \leq |2|^{-1}q^{L+n_\tau -N_\tau +1}.
		\end{equation*}
		By \eqref{4.32}, $|2|^{-1}q^{L+n_\tau - N_\tau +1} <1$ and hence, $||\mathfrak{c}^{-1}(ug_0)|| < q^L$, a contradiction and hence \eqref{4.46} is zero, proving the lemma. 
	\end{proof}
	
	\begin{lemma} We have 
		\begin{equation*}
			I_L(\chi \times \eta, g) = 0
		\end{equation*}
		for all $0 < L < n_\tau$ and
		\begin{equation*}
			I_{-L}(\chi \times \eta, g) = 0
		\end{equation*}
		for all $0 \leq L < N_\tau - n_\tau -1 - \nu_F(2)$, with $v_F$ the valuation over $F$.
	\end{lemma}
	
	The lemma is similar to \cite[Lemmas 4.5--4.6]{RS} and the proof follows from their proof along with the modification presented in proof of Lemma \ref{lem4.3}. As such, we omit the details here. 
	
	From these two lemmas, we can conclude that in \eqref{4.34}, for $\Re(s) \gg 0$,
	\begin{align*}
		Mf(\varepsilon g^{-1})
		=& \, |2|^{\frac{c}{2}(c-1)}\chi(\varepsilon^c)\eta^{1/2}(\varepsilon^c)\\ 
		&\times \int\limits_{\substack{||\mathfrak{c}^{-1}(ug_0)|| = q^{-(N_\tau - n_\tau -1 - \nu_F(2))}\\ u \in U_0}} \chi(\det(I +ug_0)) \eta^{1/2}(\det(I+ug_0)\zeta_g) |\det(I+ug_0)|^{(s-c/2)}du. \numberthis \label{4.51}
	\end{align*}
	The domain of integration in \eqref{4.51} is over $u\in U_0$, such that $ug_0 \in \mathfrak{c}(\mathfrak{g}(2^{-1}$ $\mathcal{P}_F^{N_\tau - n_\tau -1}))\subset U_0$, by \eqref{4.33}. Hence, $g_0 \in U_0$ and $MF(\varepsilon g^{-1})$ is supported in $U$. In particular, the condition of Claim \ref{claim2} is satisfied and hence $\zeta_g = 1$ over the entire support. For $g \in U$, we may change the variable $u \mapsto ug_0^{-1}$ in \eqref{4.51} and get that
	$Mf(\varepsilon g^{-1})$ is constant on $U_0$ and hence on $U$ up to $C^0_G$. As the local zeta integral is defined modulo $C^0_G$, for the purpose of computing the integral we can take $Mf(\varepsilon g^{-1})$ equal to $Mf(\varepsilon\cdot 1_G)$. From \eqref{4.51}, as $|\det(I_c+u)|=|2|^c$ for $u$ close to $I_c$, 
	\begin{align*}
		Mf(\varepsilon\cdot 1_G) =|2|^{\frac{c}{2}(2s-1)}\chi(\varepsilon^c)\eta^{1/2}(\varepsilon^c)\int\limits_{\substack{||\mathfrak{c}^{-1}(u)|| = |2|^{-1}q^{-(N_\tau - n_\tau -1)}\\ u \in U_0}} \chi(\det(I +u)) \eta^{1/2}(\det(I+u)) du. 
	\end{align*}
	By change of variables $u = \mathfrak{c}(v)$, $v\in \mathfrak{g}(2^{-1} \mathcal{P}_F^{(N_\tau - n_\tau -1)})$ and choosing $du = dv$ with $I + \mathfrak{c}(v) = 2(I+v)^{{-1}}$, we get 
	\begin{align*}
		Mf(\varepsilon \cdot 1_G) = |2|^{\frac{c}{2}(2s-1)}\chi((2\varepsilon)^c)\eta^{1/2}((2\varepsilon)^c) \int\limits_{\mathfrak{g}(2^{-1}q^{N_\tau-n_\tau-1})} \chi^{-1}(\det(I -v)) 
		\eta^{-1/2}(\det(I-v)) dv. 
	\end{align*}
	This shows the result for $\Re(s) \gg 0$, which can be analytically continued to all $s$. Hence, using that $v_1$ is $U$-invariant in \eqref{4.2} and that $\langle v_1, \widehat{v}_2 \rangle =1$, we get that 
	\begin{equation*}
		\Gamma(s, \pi \times \chi, \psi) = 	M^*(\chi \times \eta,s)f_1(\delta_\varepsilon i(1,\varepsilon)).
	\end{equation*}
	This concludes the proof for Theorem \ref{thm3.1}. \hfill $\blacksquare$

\end{document}